\documentclass[11pt]{amsart}
\usepackage{amsfonts,amssymb,amscd,xy}
\usepackage[leqno]{amsmath}
\usepackage{mathrsfs}
\usepackage{amsxtra}
\usepackage{amsthm}
\usepackage{tabularx}
\usepackage{enumitem}
\usepackage{epsfig,amssymb,xy}
\usepackage{bm} 

\usepackage{verbatim} 

\usepackage[usernames,dvipsnames]{xcolor}
\usepackage[colorlinks = true, linkcolor = blue, citecolor = blue]{hyperref} 

\makeatletter
\def\bm#1{\mathpalette\bmstyle{#1}}
\def\bmstyle#1#2{\mbox{\boldmath$#1#2$}}

\DeclareMathAlphabet{\mathcalligra}{T1}{calligra}{m}{n}

\newcommand{\s}{\mathscr}

\newcommand{\id}{\mathrm{id}} 
\newcommand{\pf}{{\rm pf}}
\newcommand{\un}{{\rm un}}
\newcommand{\ra}{{\rm ra}}

\newcommand{\A}{{\mathbb A}}
\newcommand{\Z}{{\mathbb Z}}

\newcommand{\N}{{\mathbb N}}
\newcommand{\G}{{\mathbb G}}
\newcommand{\F}{{\mathbb F}}
\newcommand{\W}{{\mathbb W}}
\newcommand{\Rg}{{\mathbb R}}
\newcommand{\spec}{\mathrm{Spec}}

\newcommand{\Hom}{{\mathrm{Hom}}}

\newcommand{\kk}{\kappa} 
\newcommand{\cO}{{\mathcal O}}

\def\e{\kern 0.08em}
\def\be{\kern -.1em}
\def\bbe{\kern -.07em}
\def\le{\kern 0.03em}
\def\lle{\kern 0.015em}
\def\lbe{\kern -.025em}
\def\llbe{\kern -.03em}

\numberwithin{equation}{section}
\newtheorem{lemma}[equation]{Lemma}
\newtheorem{theorem}[equation]{Theorem}
\newtheorem{proposition-definition}[equation]{Proposition-Definition}

\newtheorem{proposition}[equation]{Proposition}

\theoremstyle{definition}
 
\newtheorem{remark}[equation]{Remark}
\newtheorem{remarks}[equation]{Remarks}

\begin{document}

\input xy 
\xyoption{all}

\title{Greenberg algebras and ramified Witt vectors}

\author{Alessandra Bertapelle}
\author{Maurizio Candilera}
\address{Universit\`a degli Studi di Padova, Dipartimento di Matematica ``Tullio Levi-Civita'', via Trieste 63, I-35121 Padova}
\email{alessandra.bertapelle@unipd.it}
\email{maurizio.candilera@unipd.it}

\topmargin -1cm
\smallskip
\baselineskip 15pt

\begin{abstract} Let  $\cO$ be a complete discrete valuation ring of mixed characteristic and with finite residue field $\kk$. We study a natural morphism $\bm r\colon \Rg_\cO\to \W_{\cO,\kk}$ between the Greenberg algebra of $\cO$ and the special fiber of the scheme of ramified Witt vectors over $\cO$. It is a universal homeomorphism with pro-infinitesimal kernel that can be explicitly described in some cases.    \end{abstract}





\date{\today}
\maketitle
\parskip 5pt

Let $\cO$ be a complete discrete valuation ring with field of fractions $K$ of characteristic $0$ and perfect residue field $\kk$ of positive characteristic $p$. We fix a uniformizing parameter $\pi\in \cO$. It is known from \cite{gre, lip, bga} that for any $n\in \N$ one can associate a {\it Greenberg algebra} $\Rg_n$ to the artinian local ring $\cO/\pi^n\cO$, i.e., the algebraic $\kk$-scheme that represents the fpqc sheaf associated to the presheaf 
 \[\{\text{affine $\kk$-schemes} \}\to \{\cO/\pi^n\cO\text{-algebras}\}, \qquad \spec(A)\mapsto W(A)\otimes_{W(\kk)} \cO/\pi^n\cO, \]
 where $W(A)$ is the ring of $p$-typical Witt vectors with coefficients in $A$. 
 There are canonical morphisms $\Rg_n\to \Rg_m$ for $n\geq m$, and passing to the limit one gets an affine ring scheme $\Rg_\cO$ over $\kk$ such that $W(A)\otimes_{W(\kk)} \cO=\Rg_\cO(A):=\Hom_{\kk}(\spec(A),\Rg_\cO)$ for any $\kk$-algebra $A$; see \eqref{e.ra}. The Greenberg algebra $\Rg_n$ is the fundamental stone for the construction of the Greenberg realization $\mathrm{Gr}_n(X)$ of a $\cO/\pi^n\cO$-scheme $X$; this is a $\kk$-scheme whose set of  $\kk$-rational sections coincides with $X(\cO/\pi^n\cO)$ \cite{gre},\cite[Lemma 7.1]{bga}, and it plays a role in many results in Arithmetic Geometry.

Assume that $\kk$ is finite. One can define for any $\cO$-algebra $A$ the algebra $W_\cO(A)$ of {\it ramified Witt vectors with coefficients in} $A$ \cite{haz,dri,ff, acz,sch}. These algebras are important objects in $p$-adic Hodge theory. It is well-known that if $A$ is a perfect $\kk$-algebra, there is a natural isomorphism 
\begin{equation*}
  W(A)\otimes_{W(\kk)} \cO\simeq W_\cO(A)
\end{equation*} (see \cite[I.1.2]{ff}, \cite[\S 1.2]{acz}, \cite[Prop. 1.1.26]{sch}). Hence $\Rg_\cO(A)\simeq W_\cO(A) $ if $A$ is a perfect $\kk$-algebra. For a general $\kk$-algebra $A$, Drinfeld's map $u\colon W(A)\to W_\cO(A)$ induces a unique homomorphism of $\cO$-algebras $\Rg_\cO(A)\to W_\cO(A)$, functorial in $A$. Hence there is a morphism of ring schemes over $\kk$
\[\bm r\colon \Rg_\cO\to \W_\cO\times_\cO \spec(\kk),\] where  $\W_\cO$ is the ring scheme of ramified Witt vectors  over $\spec(\cO)$, that is, $W_\cO(A)=\Hom_\cO(\spec(A),\W_\cO)$ for any $\cO$-algebra $A$. It is not difficult to check that the  morphism $\bm r$ is surjective with pro-infinitesimal kernel; hence, up to taking inverse perfection, one can identify  Greenberg algebra $\Rg_\cO$ with the special fiber of the scheme of ramified Witt vectors $\W_\cO$ (Theorem \ref{t.m}).  
The morphism $\bm r$ is deeply related to
 the scheme theoretic version $\bm u$ of  Drinfeld's functor  (Proposition \ref{prop.u}) and a great part of the paper is devoted to the study of $\W_\cO$ and $\bm u$.
In the unramified case the special fiber of $\bm u$ is a universal homeomorphism and we can explicitly describe its kernel (Proposition \ref{p.perf}). The ramified case requires ad hoc constructions (Proposition \ref{p.integral}).   
All these results allow  a better understanding of the kernel of $\bm r$ (Lemma \ref{l.examples}).

\subsection*{Notation} For any morphism of $\cO$-schemes $f\colon X\to Y$ and any  $\cO$-algebra $A$ (i.e., any homomorphism of commutative rings with unit $\cO\to A$)
we write $X(A)$ for $\Hom_{\cO}(\spec(A),X)$ and $f_A$ for the map $X(A)\to Y(A)$ induced by $f$. For any $\cO$-scheme $X$, we write $X_{\kk}$ for its special fiber. If $f\colon \spec (B)\to \spec(A)$, $f^*\colon A\to B$ denotes the corresponding morphism on global sections.
 We use bold math symbols for (ramified) Witt vectors and important morphisms. 

\subsection*{Acknowledgements} We thank an anonymous referee for suggesting new references that inspired shorter proofs of the main results.

\section{Greenberg algebras}\label{s.ga}
Let $\cO$ be a complete discrete valuation ring with field of fractions $K$ of characteristic $0$ and perfect residue field $\kk$ of positive characteristic $p$. Let $\pi\in \cO$ denote a fixed uniformizing parameter and let $e$ be the absolute ramification index so that $\cO\simeq \oplus_{i=0}^{e-1}W(\kk)\pi^i$ as $W(\kk)$-modules.
Let $\W$ (respectively,  $\W_m$) denote the ring scheme of $p$-typical Witt vectors of infinite length (respectively, length $m$) over $\spec(\Z)$ and let $\W_{\kk}$ (respectively,  $\W_{m,\kk}$) be its base change to $\spec(\kk)$. 

The {\it Greenberg algebra} associated to the artinian local ring $\cO/\pi^{n}\cO$, $n\geq 1$, is the $\kk$-ring scheme $\Rg_n$ that represents the fpqc sheaf associated to the presheaf 
 \[\{\text{affine $\kk$-schemes} \}\to \{\cO/\pi^n\cO\text{-algebras}\}, \qquad \spec(A)\mapsto W(A)\otimes_{W(\kk)} \cO/\pi^n\cO; \]
it is unique up to unique isomorphism \cite[Proposition A.1]{lip}. The explicit description of $\Rg_n$ requires some work in general (we refer the interested reader to \cite{gre, lip, bga}) but is easy when considering indices that are multiple of $e$. Indeed 
 $\Rg_{me}\simeq \prod_{i=0}^{e-1}\W_{m,\kk}$  as $\kk$-group schemes and for any
$\kk$-algebra $A$ it is 
\begin{equation*}
\Rg_{me}(A)\simeq \oplus W_m(A)\pi^i\simeq W_m(A)[T]/(f_\pi(T))\simeq W_m(A)\otimes_{W_m(\kk)}\cO/\pi^{me}\cO,
\end{equation*}
where $f_\pi(T) \in W(\kk)[T]$ is the Eisenstein polynomial  of $\pi$;
see \cite[(3.6) and Remark 3.7(a)]{bga}, where $\cO$ is denoted by $R$, and \cite[Lemma 4.4]{bga} with $R'=\cO$, $R=W(\kk)$, $m=n$ and $\Rg_n$ denoted by $\mathscr R_n$. Hence the addition law on the $\kk$-ring scheme $\Rg_{me}$ is defined component wise (via the group structure of $\W_{m,\kk}$) while the multiplication depends on $f_\pi(T)$  and mixes indices.

The canonical homomorphisms $\cO/\pi^{ne}\cO\to \cO/\pi^{me}\cO , n\geq m$, induce morphisms of ring schemes $\Rg_{ne}\to \Rg_{me}$ \cite[Proposition A.1 (iii)]{lip} and the Greenberg algebra associated to $\cO$ is then defined as the affine $\kk$-ring scheme
 \begin{equation*}
 \Rg_\cO=\varprojlim \Rg_{me}
\end{equation*} (see \cite[\S 5]{bga} where $\Rg_\cO$ is denoted by $\widetilde {\s R}$). By construction $\Rg_\cO\simeq \prod_{i=0}^{e-1}\W_{\kk}$ as $\kk$-group schemes and
\begin{equation} \label{e.ra}
\Rg_\cO(A)=W(A)[T]/(f_\pi(T))=W(A)\otimes_{W(\kk)}\cO
\end{equation}
for any $\kk$-algebra $A$ \cite[(5.4)]{bga}; note that by \cite[Lemma 4.4]{bga} the hypothesis $A=A^p$ in \cite[(5.4)]{bga} is superfluous since $\varprojlim_{m\in \N} \Rg_{me}=\varprojlim_{n\in \N} \Rg_n$.
We will say that $\Rg_\cO$ is an
{\it $\cO$-algebra scheme} over $\spec(\kk)$ since, as a functor on affine $\kk$-schemes, it takes values on $\cO$-algebras.

Note that if $\cO= W(\kk)$, then $\Rg_\cO\simeq \mathbb W_{\kk}$, the $\kk$-scheme of $p$-typical Witt vectors. 

\section{Ramified Witt vectors} \label{s.w}
Let $\cO$ be a complete discrete valuation ring with field of fractions $K$ and {\it finite} residue field $\kk$ of cardinality $q=p^h$.

For any $\cO$-algebra $B$ one defines the $\cO$-algebra of {\it ramified Witt vectors} $W_\cO(B)$ as the set $B^{\N_0}$ endowed with a structure of $\cO$-algebra in such a way that the map
\begin{equation}\label{eq.phi}
\Phi_B\colon W_\cO(B)\to B^{\N_0},\qquad \bm b=(b_n)_{n\in \N_0}\mapsto \left(\Phi_0(\bm b),\Phi_1(\bm b), \Phi_2(\bm b), \dots \right),\end{equation}
is a homomorphism of $\cO$-algebras, where  $\Phi_n(\bm b)=b_0^{q^n}+\pi b_1^{q^{n-1}}+\dots +\pi^nb_n$ and the target $\cO$-algebra $B^{\N_0}$  is the product ring on which $\cO$ acts via multiplication in each component.  Proving the existence of  $W_\cO(B)$  with the indicated property requires some work; we refer to \cite{sch} for detailed proofs. 

Note that if $\pi$ is not a zero divisor in $B$ then $\Phi_B$  is injective and indeed bijective if $\pi$ is invertible in $B$.

The above construction provides a ring scheme (and in fact an $\cO$-algebra scheme) $\W_{\cO}$ such that $W_\cO(A)=\Hom_\cO(\spec(A),\W_\cO)=:\W_\cO(A)$ for any $\cO$-algebra $A$, together with a morphism of $\cO$-algebra schemes $\bm \Phi\colon \W_\cO\to \A^{\N_0}_\cO$ induced by the Witt polynomials
\begin{equation*}\label{e.phi}
\Phi_n=\Phi_n(X_0,\dots,X_n)=X_0^{q^n}+\pi X_1^{q^{n-1}}+\dots +\pi^nX_n;
\end{equation*}
more precisely, if $\A^{\N_0}_\cO=\spec(\cO[Z_0,Z_1,\dots]) $, and $\W_\cO=\spec( \cO[X_0,X_1,\dots])$ then $\bm \Phi^*(Z_n)=\Phi_n$.
Let ${\bm \Phi}_i\colon \W_\cO\to \A^1_\cO$ denote the composition of $\bm \Phi$ with the projection onto the $i$th factor.

It is $\W_{\Z_p}=\W\times_{\spec(\Z )} \spec(\Z_p)$, the base change of the scheme of $p$-typical Witt vectors over $\Z$,  but, despite the notation, $\W_{\cO}$ differs from $\W\times_{\spec(\Z )} \spec(\cO)$ in general.
\medskip

Let  $K'$ be a finite extension of $K$ with residue field $\kk'=\F_{q^r}$ and ring of integers $\cO'$ and let $\varpi\in \cO'$ be a fixed  uniformizing parameter. We can repeat the above constructions with $\varpi, q^r$ in place of $\pi,q$ and then get a morphism of  $\cO'$-algebra schemes $\bm{\Phi}'\colon \W_{\cO'}\to \A_{\cO'}^{\N_0}$ defined by the Witt polynomials
\begin{equation}\label{e.phip}
\Phi'_n (X_0,\dots,X_n)=X_0^{q^{rn}}+\varpi X_1^{q^{r(n-1)}}+\dots +\varpi^nX_n . \end{equation}  
 By \cite{dri}  there is a natural  morphism of  functors from the category of $\cO'$-algebras to the category of $\cO$-algebras \[\bm u=\bm u_{(\cO,\cO')}\colon \W_\cO \to \W_{\cO'}\qquad \ \quad (\text{{Drinfeld's functor}}) \] such that    for any $\cO'$-algebra $B$  the following diagram 
	\begin{equation*} 
  	\xymatrix{W_\cO(B)\ar[d]^{ \Phi_{B}}\ar[rr]^(0.5){ u} && W_{\cO'}(B)\ar[d]^{\Phi'_B}\\
		B^{\N_0}\ar[rr]^{\Pi'}&&B^{\N_0}
	} \end{equation*}
	commutes, where the upper arrow is induced by $\bm u$ on $B$-sections and  $\Pi'$   maps $(b_0,b_1,\dots)$ to $(b_0,b_r,b_{2r},\dots)$.  Further \[
	u([b])=[b], \quad  u(F^r\bm b)=F(u(\bm b)), \quad u(V\bm b)=\frac{\pi}{\varpi}V(u( F^{r-1}\bm b)),
	\] where $[\ ],F,V$  denote, respectively, the Teichm\"uller map, the Frobenius and the Verschiebung both in $W_\cO(B)$ and in $W_{\cO'}(B)$, and $F^r$ is the $r$-fold composition of $F$ with itself.
By construction Drinfeld's functor  behave well with respect to base change, i.e. if $\cO{''}/\cO'$ is another extension, then
\begin{equation}\label{e.ucomp}
\bm u_{(\cO,\cO{''}) }=\bm u_{(\cO',\cO{''})}\circ \bm u_{(\cO,\cO')} 
\end{equation}
as functors from the category of $\cO''$-algebras to the one of $\cO$-algebras.
More details on $\bm u$  and its scheme theoretic interpretation will be given in Section \ref{s.rrwv}.

\section{Perfection}
A $\kk$-scheme $X$ is called perfect if the absolute Frobenius endomorphism $F_X$ is an isomorphism. For any $\kk$-scheme $X $ one constructs its  {\it (inverse) perfection} $X^\pf$ as the inverse limit of copies of $X $ with $F_X$ as transition maps. It is known that the functor $(-)^\pf$ is right adjoint of  the forgetful functor from the category of perfect $\kk$-schemes to the category of $\kk$-schemes, i.e., if  $\rho\colon X^\pf\to X $ denotes the canonical projection, there is a bjection  
\begin{equation}\label{e.adj}
\Hom_{\kk}(Z , X ^\pf)\simeq \Hom_{\kk}(Z ,X ), \qquad f\mapsto \rho\circ f 
\end{equation}
for any perfect $\kk$-scheme $Z $; 
see \cite[Lemma 5.15 and (5.5)]{bga2} for more details on this. 

In the next sections we will need the following result.
\begin{lemma}\label{l.pf}
	Let $\psi\colon X\to Y$ be a morphism of $\kk$-schemes such  that $\psi_A\colon X(A)\to Y(A)$ is a bijection for any perfect $\kk$-algebra $A$. Then $\psi^\pf\colon X^\pf\to Y^\pf$ is an isomorphism and $\psi$ is a universal homeomorphism.
\end{lemma}
\begin{proof}
By  hypothesis	
\begin{equation}\label{e.zrw}
\Hom_{\kk}( Z,X)\simeq  \Hom_{\kk}(Z,Y), \qquad f\mapsto  \psi\circ f,\end{equation}
for any perfect $\kk$-scheme $Z$. In particular,
\[ \Hom_{\kk}(Y^\pf ,X^\pf)\simeq \Hom_{\kk}( Y^\pf ,X)\simeq  \Hom_{\kk}(Y^\pf,Y)\simeq \Hom_{\kk}(Y^\pf,Y^\pf) , \]
where the first and third bijections follow from \eqref{e.adj} and the second from \eqref{e.zrw}.
By standard arguments the inverse of $\psi^\pf$ is then the morphism associated with the identity on $Y^\pf$ via the above bijections. 	Consider further the following commutative square \[\xymatrix{X^\pf\ar[r]_{\psi^\pf}^\sim\ar[d]^\rho&Y^\pf \ar[d]^\rho\\
	X\ar[r]_{\psi}&Y 
} \]
Since the canonical morphisms $\rho$ are   universal homeomorphisms \cite[Rem. 5.4]{bga2}, the same is  $\psi$.
\end{proof}
 
\section{Results on ramified Witt vectors}\label{s.rrwv}
In this section we study more closely ramified Witt vectors. When possible we use a scheme theoretic approach that makes evident functorial properties and shortens proofs. 
Let notation be as in Section \ref{s.w}. 

\subsection{Frobenius, Verschiebung, Teichm\"uller maps}
In this subsection we present classical constructions  in the scheme-theoretic language. Their properties  naturally descend   from \eqref{sigma} and Remark \ref{r.unique}.

Let $B$ be an $\cO$-algebra.
If $B$ admits an endomorphism of $\cO$-algebras $\sigma$ such that $\sigma(b)\equiv b^q $ mod $\pi B$ for any $b\in B$, then the image of the homomorphism $\Phi_B$ in \eqref{eq.phi} can be characterized as follows 
\begin{equation}\label{sigma}
(a_n)_{n\in \N_0}\in {\rm Im} \Phi_B \quad \Leftrightarrow \quad
\sigma(a_n)\equiv a_{n+1} \quad \text{\rm mod}\ \pi^{n+1}B\quad \forall n\in \N_{0};
\end{equation}
see 
\cite[Prop. 1.1.5]{sch}.
We will   apply  this fact to polynomial rings $\cO[T_i, i\in I]$ with $\sigma$ the endomorphism of $\cO$-algebras mapping $T_i$ to $T^q_i$. 

\begin{lemma}\label{l.equiv} Let $\sigma\colon B\to B$ be an endomorphism of $\cO$-algebras, $\varpi\in B$ an element such that $\pi\in \varpi B$ and $f\in \N$. If $\sigma(b)\equiv b^{q^f} $ {\rm mod} $\varpi B$ for any $b\in B$, then \[
\sigma(\Phi_{fn}(\bm b))\equiv \Phi_{f(n+1)}(\bm b) \quad \text{\rm mod}\ \varpi^{fn+1}B\]
for all $\bm b=(b_0,b_1,\dots)\in W_{\cO}(B)$ and $n\geq 0$.
\end{lemma}
\begin{proof} 
	Let $\bm b=(b_0,b_1,\dots)$. 
Since $\Phi_{f(n+1)}(\bm b)\equiv \Phi_{fn}(b_0^{q^f}, b_1^{q^f},\dots) $ mod $\pi^{nf+1}B$, we are left to prove that
\[\sigma(\Phi_{fn}(\bm b))\equiv \Phi_{fn}(b_0^{q^f}, b_1^{q^f},\dots)\quad \text{\rm mod}\ \varpi^{fn+1}B.
\]
We first note that $\sigma(b)\equiv b^{q^f}$ mod $\varpi B$ implies that
\begin{equation}\label{e.equiv}
\sigma(b^{q^{s}})\equiv b^{q^{f+s}} \quad \text{\rm mod}\ \varpi^{s+1} , \quad \forall s\geq 0,
\end{equation}
(cf. \cite[Lemma 1.1.1]{sch}). 
Hence by \eqref{e.equiv}  
\begin{multline*}
\sigma(\Phi_{fn}(\bm b))=
\sigma\left(b_0^{q^{fn}}+\pi b_1^{q^{fn-1}}+\dots \pi^{fn}b_{fn}\right)=\\
\sigma(b_0)^{q^{fn}}+\pi\sigma(b_1)^{q^{fn-1}}+\dots \pi^{fn}\sigma(b_{fn})\equiv\\
b_0^{q^{(n+1)f}}+\pi b_1^{q^{f(n+1)-1}}+\dots +\pi^{fn}b_{fn}^{q^{f}}
= \Phi_{fn}(\bm b^{q^f}),
\end{multline*}
where the equivalence  holds modulo $\varpi^{fn+1}B$.
\end{proof}

Let $B$ be an $\cO$-algebra, $\sigma\colon B\to B$ an endomorphism  of $\cO$-algebras such that $\sigma(b)\equiv b^q$ modulo $\pi B$, and let $h\colon \spec(B)\to \A^{\N_0}_{\cO}=\spec (\cO[Z_0,Z_1,\dots])$ be a morphism of $\cO$-schemes; the latter is uniquely determined by $(h_0,\dots)\in B^{\N_0}$  with $h_i=h^*(Z_i)$. The morphism $h$ factors through $\bm \Phi\colon \W_\cO\to \A_\cO^{\N_0}$ if and only if  $(h_0,h_1,\dots)\in {\rm Im} \Phi_B$.
Hence   we can rephrase \eqref{sigma} as follows:
\begin{equation}\label{sigma2}
h \text{ factors through } \bm\Phi \Leftrightarrow 
\sigma(h^*(Z_n))\equiv h^*(Z_{n+1}) \quad \text{\rm mod}\ \pi^{n+1}B\quad \forall n\in \N_{0}.
\end{equation} 

\begin{remarks}\label{r.unique}
\begin{itemize}
\item[a)] Note that if $\pi$ is not a zero divisor in $B$ and $h$ factors through $\bm \Phi$, then it factors uniquely. Indeed, let $g,g'\colon \spec(B)\to \W_{\cO}$ be such that $\bm\Phi \circ g=h=\bm\Phi\circ g'$ and let $\bm b,\bm b'\in W_\cO(B)=\W_\cO(B)$ be the corresponding sections. Then $\Phi_B(\bm b)=\Phi_B(\bm b')$ and one concludes that $g=g'$ by the injectivity of $\Phi_B$ \cite[Lemma 1.1.3]{sch}.
\item[b)] Since the above constructions depend on $\pi$, it seems that one should write  $\bm \Phi_\pi$ and $\W_{\cO,\pi}$ above. However, if $\varpi$ is another uniformizing parameter of $\cO$, let $\sigma$ be the $\cO$-algebra endomorphism on $B=\cO[X_0,X_1,\dots]$ mapping $X_i$ to $X_i^q$. Then $\sigma(\Phi_{\varpi,n}(X_.))\equiv \Phi_{\varpi,n+1}(X_.))$ modulo $\varpi^{n+1}B=\pi^{n+1}B$; hence  by \eqref{sigma2} and a) one deduces the existence of a unique 
morphism $h_{\pi,\varpi}\colon \W_{\cO,\pi}\to \W_{\cO,\varpi}$ such that $\bm \Phi_\pi=\bm \Phi_{\varpi}\circ h_{\pi,\varpi}$. Similarly one constructs $h_{\varpi,\pi}\colon  \W_{\cO,\varpi}\to \W_{\cO,\pi}$ and a) implies that $h_{\varpi,\pi}\circ h_{\pi,\varpi}$ and $h_{\pi,\varpi}\circ h_{\varpi,\pi}$ are the identity morphisms.
\item[c)] Note that if 
 $ h\colon \G_\cO\to \A^{\N_0}_\cO$ is a morphism of   group (respectively, ring) schemes with $\G_\cO\simeq \A_\cO^{\N_0}$ or $\G_\cO\simeq \A_\cO^{m}$ as   schemes, and there exists a morphism $g\colon \mathbb G_\cO\to \W_\cO$, unique by point a), such that $h=\bm \Phi\circ g$, then $g$ is a morphism of group (respectively, ring) schemes. Indeed let $\mu_{\G},\mu_\W, \mu_\A$ be the group law on $\G_\cO$, $\W_\cO$ and $\A_\cO^{\N_0}$ respectively. Since $\G_\cO\times_\cO\G_\cO=\spec(C)$ with $C$ reduced, in order to prove that $g\circ\mu_\G=\mu_\W\circ (g\times g)\colon \G_\cO\times_\cO\G_\cO\to \W_\cO$, it suffices to prove that $\bm \Phi\circ g\circ\mu_\G=\bm \Phi\circ \mu_\W\circ (g\times g)$. Now 
$\bm \Phi\circ g\circ\mu_\G=h\circ \mu_\G=
 \mu_\A\circ(h\times h)=\mu_\A\circ (\bm \Phi\times \bm \Phi)\circ (g\times g)= \bm \Phi\circ \mu_\W\circ (g\times g)$. Similar arguments work for the multiplication law when considering morphisms of ring schemes.
\end{itemize}
\end{remarks}
As applications of \eqref{sigma} and \eqref{sigma2} one proves the existence of the Frobenius, Verschiebung and Teichm\"{u}ller morphisms  as well as of endomorphisms $\bm \lambda\colon \W_\cO\to \W_\cO$ for any $\lambda\in \cO$. 

We now see how to deduce the existence of classical group/ring endomorphisms of $\W_\cO$ from endomorphisms of $\A^{\N_0}_\cO$.

\begin{proposition} \label{pro.fvlt}
\begin{itemize}
\item[i)] Let $f$ be the endomorphism of $\A^{\N_0}=\spec (\cO[Z_0,Z_1,\dots])$ such that $f^*(Z_n)=Z_{n+1}$. There exists a unique morphism of ring schemes  $\bm F\colon\W_\cO\to \W_\cO $   such that $\bm \Phi\circ \bm F=f\circ \bm \Phi$.
\item[ii)]  Let $v$ be the endomorphism of $\A^{\N_0}=\spec (\cO[Z_0,Z_1,\dots])$ such that $v^*(Z_0)=0$ and $v^*(Z_{n+1})=\pi Z_n$ for $n\geq 0$. Then there exists a unique  morphism  of     $\cO$-group schemes  $\bm V\colon\W_\cO\to \W_\cO $ such that $\bm \Phi\circ \bm V=v\circ \bm \Phi$.   
\item[iii)] For $\lambda\in \cO$ let $f_\lambda$ be the group endomorphism of $\A^{\N_0}_\cO=\spec (\cO[Z_0,\dots])$ such that $f^*_\lambda(Z_n)=\lambda Z_{n}$. Then there exists  a unique morphism  of $\cO$-group schemes  $\bm \lambda\colon\W_\cO\to \W_\cO $ such that $\bm \Phi\circ \bm \lambda=f_\lambda\circ \bm \Phi$.
\item[iv)] 
Let $\sigma\colon \A^1_\cO\to \A^1_\cO=\spec(\cO[T])$ be the morphism of $\cO$-schemes such that $\sigma^*(T)= T^q$  and let $\bm \sigma=(id, \sigma,\sigma^2,\dots) \colon \A^1_\cO\to \A^{\N_0}_\cO$.
Then there exists a unique morphism of $\cO$-schemes $\bm \tau\colon \A^1_{\cO}\to \W_\cO$ such that $\bm \Phi\circ \bm \tau=\bm \sigma$. It is a multiplicative section of the projection onto the first component $\bm \Phi_0\colon \W_\cO\to \A^1_\cO$. 
\end{itemize}
	
\end{proposition}
\begin{proof}
For proving i)-iii) we use \eqref{sigma2} with $B=\cO[X_0,X_1,\dots]$, the ring  of global sections of $\W_\cO$,  endowed with its unique lifting of  Frobenius, more precisely with the morphism  of $\cO$-algebras  $\sigma$ mapping $X_i$ to $X_i^q$. Let $\bm X$ denote   the vector  $(X_0,X_1,\dots)\in W_\cO( \cO[X_0,X_1,\dots])$ and   and set $ {\bm X}^\sigma=(X_0^q,X_1^q,\dots)$. 

The morphism $\bm F$ exists as soon as the condition in \eqref{sigma2} is satisfied for $h= f\circ\bm \Phi$ , i.e., if  $\Phi_{n+1}({\bm X}^\sigma)\equiv \Phi_{n+2}(\bm X)$ modulo $\pi^{n+1}$ for any $n$. This is evident since $\Phi_{n+2}(\bm X)=\Phi_{n+1}({\bm X}^\sigma)+\pi^{n+2}X_{n+2}$.

The morphism $\bm V$ exists as soon as the condition in \eqref{sigma2} is satisfies for $h=v\circ \bm \Phi$, i.e., if $0\equiv \pi X_0$ modulo $\pi B$ and $\pi\Phi_{n-1}({\bm X}^\sigma)\equiv \pi\Phi_{n}(\bm X)$ modulo $\pi^{n+1}B$ for any $n\geq 1$. The first fact is trivial while the second is evident since $\Phi_{n}({\bm X}^\sigma)=\Phi_{n-1}(\bm X^q)+\pi^{n}X_{n}$.

The morphism $\bm \lambda$ exists as soon as
the condition in \eqref{sigma2} is satisfies for $h=f_\lambda \circ \bm \Phi$, i.e., if 
 $\lambda\Phi_{n}({\bm X}^\sigma)\equiv \lambda\Phi_{n+1}(\bm X)$ modulo $\pi^{n+1}$ for any $n$. This is evident since $\Phi_{n}({\bm X}^\sigma)\equiv \Phi_{n+1}(\bm X)$ modulo $\pi^{n+1}$  by Lemma \ref{l.equiv} with $f=1, \varpi=\pi$.

Uniqueness of $\bm F, \bm V, \bm \lambda$ follows  by Remark \ref{r.unique} a). The fact that they are group/ring scheme morphisms  follows  by Remark \ref{r.unique}   c).

For iv),   we consider  condition \eqref{sigma2} for $B=\cO[T]$ and $h=\bm \sigma$. It is satisfied since  $h^*(Z_n)=T^{q^n}$; whence $\bm\tau$ exists.  Uniqueness follows again by Remark \ref{r.unique} a) and multiplicativity of $\bm \tau$  follows from multiplicativity of $\bm \sigma$  as in Remark \ref{r.unique} c). Finally, by construction, $\bm \tau$ is a section of $\bm \Phi_0$.
\end{proof}

The ring scheme endomorphism $\bm F$ is called  {\it Frobenius} and the $\cO$-scheme endomorphism $\bm V$ is called {\it Verschiebung}.
By a direct computation one checks that for any $\cO$-algebra $A$, the induced homomorphism $F_A\colon W_\cO(A)\to W_\cO(A)$ satisfies 
\begin{equation}\label{e.Fp}
F_A(a_0,a_1,\dots)\equiv (a_0,a_1,\dots)^q\quad (\text{mod } \pi W_\cO(A)),
\end{equation}
and, if $A$ is a $\kk$-algebra, 
\begin{equation}\label{e.Fp2}
F_A(a_0,a_1,\dots)= (a_0^q,a_1^q,\dots) 
\end{equation}
holds. Further, both ${ F}_{\!A}$ and ${V}_{\!A}$ are $\cO$-linear \cite[Sect. 1]{sch} and 
\begin{eqnarray}
& & F_AV_A=\pi\cdot {\id_{W_\cO(A)}},\label{e.FVVF} \\
& & V_AF_A= \pi\cdot {\id_{W_\cO(A)}}, \qquad \text{ if } \pi A=0, \label{e.FVVF2}\\
\nonumber & &\bm a 	\cdot V_A(\bm c)=V_A(F_A(\bm a)\cdot \bm c), \quad \text{for all}\quad \bm a,\bm c\in W_\cO(A). 
\end{eqnarray}
Finally  ${V}_{\!A}^nW_\cO(A)$ is an ideal of $W_\cO(A)$ for any $n>0$ where ${V}_{\!A}^n$ denotes the $n$-fold composition of ${V}_{\!A}$.
Note that $W_\cO(A)=\varprojlim_{n\in \N} W_\cO(A)/{V}_{\!A}^nW_\cO(A)$ and if $A$ is a semiperfect $\kk$-algebra, i.e., the Frobenius is surjective on $A$, then ${V}_{\!A}^nW_\cO(A)=\pi^n W_\cO(A)$. 

The morphism $\bm\tau$ is called {\it Teichm\"{u}ller map}. For any $\cO$-algebra $B$, we have $\tau_B\colon B\to W_\cO(B), b\mapsto [b]:=(b,0,0,\dots)$, since $\Phi_B([b])=(b,b^q,b^{q^2},\dots)$.
Note that $\bm \sigma$ is not a morphism of $\cO$-group schemes and hence we can not expect that $\bm\tau$ is a morphism of group schemes.

\begin{remark}\label{r.glue}
For any subset $I \subset \N_0$ and any $\cO$-algebra $A$, let $W_{\cO,I}(A)$ denote the subset of $W_{\cO}(A)$ consisting of vectors $\bm b=(b_0,\dots)$ such that $b_i=0$ if $i\notin I$. 
If $J\subset \N_0$ satisfies $I\cap J=\varnothing$, then the sum in $W_\cO(A)$ of a vector $\bm b= (b_0,\dots)\in W_{\cO,I}(A)$ and a vector  $\bm c= (c_0,\dots)\in W_{\cO,J}(A)$ is simply obtained by "gluing" the two vectors, i.e., $\bm b+\bm c=\bm d=(d_0,\dots)\in W_{\cO,I\cup J}(A)$ with $d_i=b_i$ if $i\in I$ and $d_i=c_i$ if $i\in J$. For proving this fact, since $A$ can be written as quotient of a polynomial algebra over $\cO$ with possibly infinitely many indeterminates, we may assume that $A$ is $\pi$-torsion free. In this case $\bm d$ is uniquely determined by the condition  $\sum_{i=0}^{n}\pi^id_i^{q^{n-i}}=\Phi_n(d_0,\dots,)=\Phi_n(b_0,\dots)+\Phi_n(c_0,\dots) =\sum_{i=0}^{n}\pi^ib_i^{q^{n-i}}+\sum_{i=0}^{n}\pi^ic_i^{q^{n-i}}$; since for any index $i$ either $b_i$ or $c_i$ (or both) is zero, the above choice of $d_i$ works.
More generally, if $I_0,\dots, I_r$, are disjoint subsets of $\N_0$, and $\bm b_j$ are vectors in $W_{\cO,I_j}(A)$, then the sum $\bm b_0+\dots +\bm b_r$ is obtained by "gluing" the vectors $\bm b_j$. 
As immediate consequence, 
any element in $ W_\cO(A)$ can be written as 
\begin{equation}\label{eq.va}
(a_0,a_1,\dots)=\sum_{i=0}^\infty V_A^i[a_i], 
\end{equation} 
since $V^i[b] =(0,\dots,0,b,0,\dots)\in W_{\cO,\{i\}}$.\end{remark}
 
\begin{lemma}\label{l.wn}
Let $B$ be a $\kk$-algebra and consider the map
\[ B^n\to W_{\cO,n}(B):= W_\cO(B)/V^n_BW_\cO(B),\qquad (b_0,\dots,b_{n-1})\mapsto \sum_{j=0}^{n-1} [b_j]\pi^j.
\] 
 If $B$ is reduced (respectively, semiperfect, perfect) the above map is injective (respectively, surjective, bijective).
Hence if $B$ is semiperfect (respectively, perfect), any element of $W_\cO(B)=\varprojlim W_{\cO,n}(B)$
can be written (respectively, uniquely written) in the form $\sum_{j=0}^\infty [b_j]\pi^j$.
\end{lemma}
\begin{proof}
By \eqref{e.FVVF}, \eqref{e.FVVF2}, \eqref{e.Fp} and Remark \ref{r.glue} it is \[\sum_{j=0}^{n-1} [b_j]\pi^j=\sum_{j=0}^{n-1} \pi^j[b_j]=\sum_{j=0}^{n-1} V^jF^j[b_j] =\sum_{j=0}^{n-1} V^j [b_j^{q^j}]=(b_0,\dots, b_{n-1}^{q^{n-1}},0,\dots),
\] where we have omitted the subscript $B$ on $F$ and $V$. Injectivity is clear when $B$ is reduced. Assume now $B$ semiperfect and let $\bm b=(b_0,b_1,\dots )\in W_\cO(B)$. Then by Remark  \ref{r.glue} 
\[\bm b=(b_0,\dots,b_{n-1},0,0,\dots)+(0,\dots,0,b_n,\dots)\in (b_0,\dots,b_{n-1},0,0,\dots)+V^nW_{\cO}(B),
\] and by \eqref{eq.va} \& \eqref{e.Fp2}
\[
(b_0,\dots,b_{n-1},0,\dots)=\sum_{j=0}^{n-1}V^i[b_i]=\sum_{j=0}^{n-1}V^iF^i[b_i^{1/q^i}]=\sum_{j=0}^{n-1}\pi^i[b_i^{1/q^i}]
\] where $b_i^{1/q^i}$ denotes any $q^i$th root of $b_i$, which exists since $B$ is semiperfect. Hence surjectivity is clear too.
\end{proof}

\subsection{The Drinfeld morphism} 
Let $K'$ denote a finite extension of $K$ with residue field $\kk'=\F_{q^r}$, ring of integers $\cO'$ and ramification degree $e$; since we don't work with absolute ramification indices in this section, there is no risk of confusion with notation of Section \ref{s.ga}.
Let $\varpi\in \cO'$ be a uniformizing parameter and write $\pi=\alpha\varpi^e$ with $\alpha$ a unit in $\cO'$. Let 
$\Phi'_n (X_0,\dots,X_n)=X_0^{q^{rn}}+\varpi X_1^{q^{r(n-1)}}+\dots +\varpi^{n}X_n$ be  the polynomials as in \eqref{e.phip} that define the morphism $\bm \Phi'\colon \W_{\cO'}\to \A^{\N_0}_{\cO'}$.

\begin{proposition}\label{prop.u} 
There exists a unique morphism  of $\cO'$-ring schemes
$\bm u=\bm u_{(\cO,\cO')} $ such that the following diagram
\begin{equation}\label{d.square}
\xymatrix{\W_\cO\times_\cO \spec\, \cO'\ar[d]^{\bm \Phi\times {\rm id}_{\cO'}}\ar[rr]^(0.6){\bm u} && \W_{\cO'}\ar[d]^{\bm{\Phi}'}\\
	\A^{\N_0}_{\cO'}\ar[rr]^{\Pi'}&&\A^{\N_0}_{\cO'}
}\end{equation}
commutes, where $\bm \Phi\times {\rm id_{\cO'}}$ is the base change of $\bm\Phi$ to $\spec(\cO')$ and $\Pi'$ is the morphism mapping $(x_0,x_1,\dots)$ to $(x_0,x_r,x_{2r},\dots)$. 
For any $\lambda\in \cO$ it is $\bm \lambda\circ \bm u=\bm u \circ( \bm \lambda\times {\rm id}_{\cO'})$, i.e., $\bm u$ induces   homomorphisms of $\cO$-algebras $u_B\colon W_\cO(B)\to W_{\cO'}(B)$ for any $\cO'$-algebra $B$.
\end{proposition}
\begin{proof}(Cf. \cite[Prop. 1.2]{dri}.)
Let $B=\cO'[X_0,\dots] $ be the ring of global sections of $\W_{\cO}\times_\cO \spec \cO'$ and let $\sigma $ be the endomorphism the $\cO'$-algebra $B$ mapping $X_i$ to $X_i^{q^r}$. Let $ h=\Pi'\circ (\bm\Phi\times {\rm id_{\cO'}})\colon \spec(B)\to \A^{\N_0}_{\cO'}$. Then by \eqref{sigma2} the morphism of $\cO'$-schemes $\bm u$ exists as soon as $\sigma(h^*(Z_n))\equiv h^*(Z_{n+1})$ modulo $ \varpi^{n+1}B$.
By definition of $h$, this condition is equivalent to $\sigma(\Phi_{nr} )\equiv\Phi_{(n+1)r} $ modulo $ \varpi^{n+1}B$, and the latter holds by Lemma \ref{l.equiv} with $f=r$ and $\bm b=(X_0,X_1,\dots)\in W_ {\cO'}(B)$.
Hence $\bm u$ exists as morphism of schemes. Uniqueness follows by Remark \ref{r.unique} a). 
Since $\bm \Phi\times {\rm id}_{\cO'}$ and $\Pi'$ are morphism of ring schemes, the same is $\bm u$ by the commutativity of \eqref{d.square} and Remark \ref{r.unique} c).

Finally, since both $\bm \lambda\circ \bm u$ and $\bm u \circ ( \bm \lambda\times {\rm id}_{\cO'})$ correspond to the endomorphism of $	\A^{\N_0}_{\cO'}$ mapping $Z_n$ to $\lambda Z_{rn}$ on algebras, the result is clear.
\end{proof}

The morphism $\bm u$ is called the {\it Drinfeld morphism}. Note that the commutativity of \eqref{d.square} says that for any $\cO'$-algebra $B$ and any $\bm b\in W_\cO(B)$ it is
\[\Phi_{n}'(u_B(\bm b))=\Phi_{nr}(\bm b).
\]

\begin{lemma} \label{l.utau}
Let $\bm \tau, \bm \tau'$ be the Teichm\"{u}ller maps of $\W_\cO, \W_{\cO'}$ respectively. Then
$\bm \tau'=\bm u\circ (\bm \tau\times {\rm id}_{\cO'})$.
\end{lemma}
\begin{proof}
	Let $\A^1_\cO=\spec\ \cO[T]$ and $A^{\N_0}_\cO=\spec\ \cO[Z_0,Z_1,\dots]$.
Let $\bm \sigma \colon \A^1_\cO\to \A^{\N_0}_\cO$ be the morphism in Proposition \ref{pro.fvlt} mapping $Z_n$ to $T^{q^n}$ on algebras,  and let $\bm \sigma' \colon \A^1_{\cO'}\to \A^{\N_0}_{\cO'}$ be the analogous morphism  $\cO'$ mapping $Z_n$ to $T^{q^{rn}}$ on algebras. Then $\bm\tau'$ is uniquely determined by the property $\bm\Phi'\circ\bm\tau'=\bm\sigma'$.
Since $\bm\Phi'\circ \bm u\circ (\bm \tau\times {\rm id}_{\cO'})=\Pi'\circ(\bm\Phi\times {\rm id}_{\cO'})\circ (\bm \tau\times {\rm id}_{\cO'})= \Pi'\circ(\bm\sigma \times {\rm id}_{\cO'})=\bm\sigma'$, the conclusion follows.
\end{proof}

Let $B$ be an $\cO'$-algebra $B$. 
As a consequence of the above lemma and $\cO$-linearity of the Drinfeld map $u_B$, it is $u_{ B}(\sum_{i=0}^n[b_i]\pi^i)=\sum_{i=0}^n[b_i]\pi^i$ and hence
\begin{equation}\label{e.ubpi}
u_{ B}\big(\sum_{i=0}^\infty[b_i]\pi^i\big)=\sum_{i=0}^\infty[b_i]\pi^i,
\end{equation}
where $[b_i]$ in the left-hand side (respectively, in the right-hand side)   is the Teichm\"{u}ller representative of $b_i$ in $W_\cO(B)$ (respectively, in $W_{\cO'}(B)$) and $\pi$ in the right-hand side is viewed as element of $\cO'$.

\begin{lemma} Let $\bm F, \bm F'$ be the Frobenius maps on $\W_\cO$ and $ \W_{\cO'}$ respectively. Then
	$\bm u\circ (\bm F^r\times {\rm id}_{\cO'})=\bm F'\circ \bm u$, where $\bm F^r$ is the $r$-fold composition of $\bm F$.
\end{lemma}
\begin{proof}
Let $f$ also denote the endomorphism of $\A^{\N_0}_{\cO'}=\spec(\cO'[Z_0,Z_1,\dots])$ associated to $\bm F'$ as in Proposition \ref{pro.fvlt} i), which maps $Z_n$ to $Z_{n+1}$ on algebras, and let $f^r$ denote the $r$-fold composition of $f$.  Since $\bm \Phi'\circ\bm F'\circ \bm u=
f\circ \bm \Phi'\circ\bm u=
f\circ \Pi'\circ  (\bm \Phi\times 
{\rm id}_{\cO'})=
\Pi'\circ f^r\circ  (\bm \Phi\times 
{\rm id}_{\cO'})=
\Pi'\circ(\bm \Phi\times 
{\rm id}_{\cO'})\circ (\bm F^r\times {\rm id}_{\cO'})=
\bm \Phi'\circ \bm u\circ (\bm F^r\times {\rm id}_{\cO'})$, the conclusion follows by Remark \ref{r.unique} a).
\end{proof}

\begin{lemma}\label{l.uvf} Let $\bm{\frac{ \pi}{  \varpi}}$ denote the group homomorphism of $\W_{\cO'}$ associated with $\frac{ \pi}{ \varpi}\in \cO'$ as in Proposition \ref{pro.fvlt} iv).
Then $\bm u\circ( \bm V\times 
{\rm id}_{\cO'})=\bm{\frac{ \pi}{  \varpi}}\circ\bm V' \circ \bm u\circ (\bm F^{r-1}\times {\rm id}_{\cO'})$.
\end{lemma}
\begin{proof}
We keep notation as in Proposition \ref{pro.fvlt}: $v$ is the endomorphism  of the affine space $\A^{\N_0}_\cO$ associated with $\bm V$, similarly for $v',\bm V'$ over $\cO'$; $f$ is the endomorphism associated with $\bm F$ and $f_{ \frac{ \pi}{  \varpi}}$ the one associated with $\bm{\frac{ \pi}{  \varpi}}$.

Note that $\Pi'\circ (v\times 
	{\rm id}_{\cO'})$ maps $Z_0$ to $0$ and $Z_{n}$ to $\pi Z_{rn-1}$ if $n>0$. 
	Now \[\bm \Phi'\circ \bm u\circ( \bm V\times 
	{\rm id}_{\cO'})= \Pi'\circ  (\bm \Phi\times 
	{\rm id}_{\cO'})\circ (\bm V\times 
	{\rm id}_{\cO'}) = \Pi'\circ (v\times 
	{\rm id}_{\cO'})\circ  (\bm \Phi\times 
	{\rm id}_{\cO'}).
	\]
On the other hand, 
\[\bm \Phi'\circ\bm{\frac{ \pi}{  \varpi}}\circ\bm V' \circ \bm u  =
f_{ \frac{ \pi}{  \varpi}}\circ\bm  \Phi'\circ \bm V'\circ  \bm u=
f_{ \frac{ \pi}{  \varpi}}\circ v'\circ \bm \Phi'\circ \bm u=
f_{ \frac{ \pi}{  \varpi}}\circ v'\circ \Pi'\circ  (\bm \Phi\times {\rm id}_{\cO'}).
\] 
Hence 
 \[\bm \Phi'\circ\bm{\frac{ \pi}{  \varpi}}\circ\bm V' \circ \bm u\circ (\bm F^{r-1}\times {\rm id}_{\cO'}) =
 f_{ \frac{ \pi}{  \varpi}}\circ v'\circ \Pi'\circ (f^{r-1}\times {\rm id}_{\cO'})\circ  (\bm \Phi\times 
 {\rm id}_{\cO'}).
 \]  Since both $ \Pi'\circ (v\times 
{\rm id}_{\cO'})$ and $
f_{ \frac{ \pi}{  \varpi}}\circ v'\circ \Pi'\circ (f^{r-1}\times {\rm id}_{\cO'})$ induce the endomorphism of $\cO'[Z_0,Z_1,\dots]$ mapping $Z_0$ to $0$ and $Z_n$ to $\pi Z_{rn-1}$ for $n>0$, they coincide and the conclusion follows  by Remark \ref{r.unique} a).
\end{proof}

We now discuss properties of the Drinfeld morphism.

\begin{lemma}\label{l.inj}
	Let $B$ be a reduced $\kk'$-algebra. Then Drinfeld morphism  induces an injective map $u_B\colon W_\cO(B)\to W_{\cO'}(B)$  on $B$-sections.
\end{lemma}

\begin{proof}
	Let $B^\pf$ denote the perfect closure of $B$, i.e., $B^\pf=\varinjlim_{i\in \N_0} B_i$ with $B_i=B$ and Frobenius $b\mapsto b^p$ as transition maps. Since $B$ is reduced, the canonical map $\phi\colon B=B_0\to B^\pf$ is injective and thus the same is $\W_\cO(\phi)$. Hence, it suffice to consider the case where $B$ is perfect. By Lemma \ref{l.wn},
	any element $\bm b$ of $W_\cO(B)$ is of the form $\sum_{i=0}^\infty[b_i] \pi^i, b_i\in B$, and hence $u_B(\bm b)=\sum_{i=0}^\infty [b_i]\pi^i$ by \eqref{e.ubpi}. Injectivity of $u_B$ is then clear since $\pi^i\in (V'_B)^{ei}W_{\cO'}(B)$ and $W_{\cO'}(B)$ has no $\pi$-torsion.
\end{proof}

Note that if $\cO'\neq \cO$ and $B$ is a non-reduced $\kk'$-algebra, then $u_B$ is not injective. Indeed let $0\neq b\in B$ such that $b^p=0$. Then by \eqref{e.FVVF} \& \eqref{e.FVVF2} with $\cO'$ in place of $\cO$ and Lemmas \ref{l.utau} and \ref{l.uvf} we have 
\[u_B(V_B[b])=\alpha\varpi^{e-1} V'_B([b])
=\alpha (V'_B)^e(F'_B)^{e-1}([b])=\alpha(V'_B)^e(0)=0
\] if $r=1$ and $e>1$, and $u_B(V_B[b])=\alpha\varpi^{e-1} V'_B(u_B( 0))=0$ if $r>1$.

More precise statements can be given in the unramified or totally ramified cases.

\subsubsection{The unramified case}

\begin{lemma}\label{l.ub}
	Let $\cO'/\cO$ be an unramified extension and let $B$ be a $\kk'$-algebra. Then $u_B\colon W_\cO(B)\to W_{\cO'}(B)$ is injective (respectively, surjective, bijective) if $B$ is reduced (respectively, semiperfect, perfect).
\end{lemma}
\begin{proof}
	Let $\W_\cO=\spec \ \cO[X_0,X_1,\dots]$, $\W_{\cO'}=\spec \ \cO'[Y_0,Y_1,\dots]$ and set $u_i=u^*(Y_i)\in \cO'[X_0,\dots,]$, so that $\Phi'_m(u_0,u_1,\dots)=\Phi_{mr}(X_0,X_1,\dots)$ by commutativity of \eqref{d.square}. We   claim that
	\[ u_0=X_0,  \qquad u_m\equiv X_m^{q^{m(r-1)}}  \quad \text{  mod  } (\pi)  \quad \text{ for  }    m>0.  \]
	Since $u_0=X_0$ is clear by construction, only the second equivalence has to be proved.
	We proceed by induction on $m$. First note that  for any $m\geq 0$
	\[\Phi_{(m+1)r}(X_0,\dots)\equiv X_0^{q^{(m+1)r}}+\dots+\pi^m X_{m}^{q^{(m+1)r-m}}+\pi^{m+1}X_{m+1}^{q^{(m+1)r-m-1}} \text{ mod } (\pi^{m+2})\]
	and 
	\[\Phi_{m+1}'(Y_0,\dots)=Y_0^{q^{ (m+1)r}}+\dots+\pi^m Y_m^{q^{r}} +\pi^{m+1}Y_{m+1 }.\] 
	Assume that  $u_i\equiv X_i^{q^{i(r-1)}}$ mod $(\pi)$ for $0\leq i\leq m$, then 
	\[\pi^{i}u_i^{q^{(m+1-i)r}}   \equiv  \pi^{i}X_{i}^{q^{(m+1)r-i}}\text{ mod } (\pi^{i+1+(m+1-i)r}),
	\]
	where $i+ 1+(m+1-i)r>i+1+m+1\geq m+2$. Hence \[ 0=\Phi_{m+1}'(u_0,\dots)-\Phi_{(m+1)r}(X_0,\dots)
	\equiv 
	\pi^{m+1}u_{m+1 } -\pi^{m+1}X_{m+1}^{q^{(m+1)r-m-1}} \text{ mod } (\pi^{m+2}),\]
	thus the claim.
	
	Now, if $B$ is any $\cO$-algebra, $\bm b=(b_0,\dots)\in W_\cO(B)$ and $u_B(\bm b)=\bm c=(c_0,c_1,\dots)$ it is $c_0=b_0$ and $c_m\equiv b_m^{q^{m(r-1)}}$ mod $ \pi B$. In particular, if $B$ is a $\kk'$-algebra, it is 
	\begin{equation}\label{e.cb}c_m= b_m^{q^{m(r-1)}},  \qquad \forall m\geq 0.\end{equation} This implies that $u_B$ is injective if $B$ is reduced (as already seen in Lemma \ref{l.inj}), surjective if $B$ is semiperfect and bijective if $B$ is perfect. \end{proof}

The above lemma has the following geometric interpretation. 

\begin{proposition}\label{p.perf}
	Assume that the extension $\cO'/\cO$ is unramified. Then Drinfeld's morphism $\bm u$ restricted to special fibers is a universal homeomorphism with pro-infinite\-simal kernel isomorphic to $\spec(\kk'[X_0,X_1,\dots]/(X_0, \dots,X_i^{q^{i(r-1)}}\!\!, \dots)$ where $q^r$ is the cardinality of $\kk'$.
\end{proposition}
\begin{proof} The first assertion follows from  Lemmas \ref{l.pf}  and \ref{l.ub}. By  the very explicit description of $\bm u_{\kk}$ in \eqref{e.cb}  one gets  the   assertion on the kernel.
\end{proof}

\subsubsection{The totally ramified case}
Let $\cO'/\cO$ be a totally ramified extension of degree $e>1$. Then  $\kk'=\kk$,  $\cO'=\oplus_{i=0}^{e-1}\cO\varpi^i$ as $\cO$-module, and $\pi=\alpha\varpi^e$ with $\alpha$ a unit in $\cO'$.
	Let $B$ be  a $\cO'$-algebra.
We can not expect $u_B\colon W_\cO(B)\to W_{\cO'}(B)$ to be surjective, even if $B$ is a  perfect $\kk$-algebra; indeed \eqref{e.ubpi} shows that $\varpi$ is not in the image of $u_B$.  
Note that   $u_B$ is a morphism of $\cO$-algebras  and hence we can extend it to a morphism of $\cO'$-algebras
\begin{equation}\label{e.urab}
u^{\ra}_B=u_B\otimes {\id}\colon W_{\cO}(B)\otimes_\cO \cO'\to W_{\cO'}(B), \qquad \sum_{i=0}^{e-1} \bm b_i\otimes \varpi^i\mapsto \sum_{i=0}^{e-1} u_B(\bm b_i) \varpi^i,
\end{equation} 
with $\bm b_i\in W_\cO(B)$. Since  
for any $\cO$-algebra $A$ it is 
\begin{equation}\label{e.waw}
W_{\cO}(A)\otimes_\cO \cO'= W_{\cO}(A)\otimes_{\cO} \oplus_{i=0}^{e-1}\cO \varpi^i=\oplus_{i=0}^{e-1} W_{\cO}(A)  \varpi^i,
\end{equation} 
forgetting about the multiplication on $W_\cO(B)\otimes_\cO\cO'$,  $ u^\ra_B $ is the group homomorphism making the following diagram commute
{\small	\begin{equation}\label{d.double}
	\xymatrix{\prod_{i=0}^{e-1}\W_{\cO }(B)\ar[ddd]^{\prod  \Phi_B  }\ar[rrr]^(0.6){  u^\ra_B} &&&  \W_{\cO'}(B)\ar[ddd]^{ \Phi'_B }\\
		&	(\bm b_i)_i \ar@{|->}[d] \ar@{|->}[r]  & \sum_i u_B(\bm b_i)\varpi^i  \ar@{|->}[d]&\\
		&( \Phi_B(\bm b_i))\ar@{|->}[r] &   \sum_i  \Phi_B ( \bm b_i )\varpi^i= \sum_i  \Phi'_B( u_B(\bm b_i))\varpi^i & \\
		\prod_{i=0}^{e-1}	B^{\N_0}  \ar[rrr]  &&&B^{\N_0} 
	}
	\end{equation}}
 We deduce from \eqref{e.waw} that the
product group scheme $ \prod_{i=0}^{e-1} \W_{\cO}$, whose group of $A$-sections is $ \oplus_{i=0}^{e-1} W_{\cO}(A)$, for any $\cO$-algebra $A$,    can  be endowed  with a ring scheme structure that depends on the Eisenstein polynomial of $\varpi$ and mixes components. We denote by $\prod^\varpi  \W_{\cO }$ the resulting ring scheme over $\cO$. 
In particular the functoriality of the maps $u_B^\ra$ say the existence of a morphism of ring schemes over $\cO'$
\[\bm u^\ra \colon \prod^\varpi  \W_{\cO } \times_\cO\spec (\cO')\to \W_{\cO'}
\]
which induces $u^\ra_B$ on $B$-sections. More precisely,  $\bm u^\ra$ is a morphism of schemes of $\cO'$-algebras. 
 Let 
 \begin{equation}\label{e.rak}
 \bm u^{\ra}_{\kk}\colon \prod^\varpi \W_{\cO,\kk}\to \W_{\cO',\kk}.
\end{equation}
be the restriction of $\bm u^{\ra}$ to   special fibers. 

 We can not expect that results in Lemma \ref{l.ub} and Proposition \ref{p.perf}   hold in the totally ramified case, but they hold for $\bm u^{\ra}$ in place of $\bm u$.

\begin{lemma}\label{l.ue} 
	Let $\cO'/\cO$ be a totally ramified extension of degree $e$ and let $B$ be a $\kk$-algebra. If $B$ is reduced (respectively, semiperfect, perfect) then 
	the homomorphism $u^{\ra}_B=u_B\otimes {\id}$   in \eqref{e.urab} is injective (respectively, surjective, bijective).
\end{lemma}
\begin{proof}
	For the injectivity,  as in the proof of Lemma \ref{l.inj}, we may assume that $B$ is perfect. Let $\bm x=\sum_{i=0}^{e-1}\bm b_i\otimes \omega^i$ with $\bm b_i=\sum_{j=0}^\infty [b_{i,j}]\pi^j\in W_\cO(B)$ by Lemma \ref{l.wn}.
	Then by \eqref{e.ubpi} it is  $u^{\ra}_B(\bm x)=\sum_{i=0}^{e-1}\sum_{j=0}^\infty [b_{i,j}]\pi^j\varpi^i=\sum_{i=0}^{e-1}\sum_{j=0}^\infty \alpha_j (V'_B)^{ej+i}[\tilde b_{i,j}] $ with $\alpha =\pi/\varpi^e$ a unit in $\cO'$, $\tilde b_{i,j}$ the $q^{ej+i}$th power of $b_{i,j} $ and $V_B'$ the Verschiebung on $W_{\cO'}(B)$. 
 Hence injectivity follows. 
	
	Now we prove surjectivity in the case  where $B$ is semiperfect. 
	By Lemma \ref{l.wn}   any element of $W_{\cO'}(B)$ can be written in the form \[
	\sum_{j=0}^{\infty} [a_{j}]\varpi^j
	=\sum_{i=0}^{e-1}\sum_{h=0}^{\infty} [a_{he+i}]\pi^h\varpi^i/\alpha^h
	\] 
	It suffices to check that $\sum_{h=0}^{\infty} [a_{he+i}]\pi^h\alpha^{-h}$ is in the image of $u_B^{\ra}$ for all $i$. 
	Note that the series $\sum_{h=0}^{\infty} [a_{he+i}]\otimes\pi^h \alpha^{-h}$ is  in $W_{\cO}(B)\otimes_{\cO }\cO'$ since
	\begin{multline*}W_{\cO}(B)\otimes_{\cO }\cO'\simeq
	\big(\varprojlim_m W_{\cO}(B)/\pi^mW_\cO(B)\big)\otimes_{\cO }\cO'\simeq\\
	\varprojlim_m \big((W_{\cO}(B)/\pi^mW_\cO(B))\otimes_{\cO }\cO'\big)=\varprojlim_m W_{\cO}(B)\otimes_{\cO }\cO' /\pi^m\left(W_\cO(B)\otimes_{\cO }\cO'\right),
	\end{multline*} 
	where the first isomorphism follows by Lemma \ref{l.wn} and the second by the fact that $\cO'$ is a finite free $\cO$-module. Now by 
	$\cO'$-linearity of $u_B^{\ra}$ and Lemma \ref{l.utau}
	\[u_B^{\ra}\big(\sum_{h=0}^{\infty} [a_{he+i}]\otimes\pi^h \alpha^{-h} \big)=\sum_{h=0}^{\infty} [a_{he+i}] \pi^h \alpha^{-h},
	\]
	and we are done.
\end{proof}

 We now study morphisms $\bm u$ and $\bm u^\ra$.

\begin{proposition}\label{p.integral}
	Let $\cO'/\cO$ be a totally ramified extension of degree $e$. Then  the morphism
	$\bm u^{\ra}_{\kk} $ in \eqref{e.rak} is a universal homeomorphism with pro-infinitesimal kernel isomorphic to  \[\spec(\kk[X_{n,i}; n\in \N_0, 0\leq i<e]/( X_{n,i}^{q^{n(e-1)+i}} ).\]   
\end{proposition}
\begin{proof}
The first assertion follows from Lemmas \ref{l.pf} and \ref{l.ue}.  
	
We now describe the kernel of the morphism of $\cO'$-group schemes
\[\bm u^{\ra}\colon \prod^\varpi \W_{\cO'}=\spec \ \cO'[X_{n,i}, n\in \N_0, 0\leq i<e,]\stackrel{ }{\longrightarrow}\W_{\cO'}=\spec \ \cO'[Y_0,Y_1,\dots].
\]	Set $u_{m }^\ra=u^{\ra *}(Y_m)\in \cO'[X_{n,i}, n\in \N_0, 0\leq i<e,]$ where $u^{\ra *}$ is the homomorphism  induced by $  \bm u^\ra$  on global sections. The kernel of $\bm u^\ra$ is the closed subscheme of $\prod ^\varpi \W_\cO$  whose ideal $I$ is generated by the polynomials $u^{\ra}_m, m\geq 0$. Let $J$ be the ideal generated by the monomials $X_{n,i}^{q^{n(e-1)+i}}$. 
We want  to prove that
$I$   coincides with   $J$  modulo $\varpi$.
Both ideals admit a filtration by subideals  $I_s\subset I$, $J_s\subset J$ where $I_s$ is generated by those $u^{\ra}_m$ with $ m\leq s$ and $J_s$ is generated by monomials $X_{m,j}^{q^{m(e-1)+j}}$ such that $me+j\leq s$. 
It is sufficient to check that $I_s$ coincides with $J_s$ modulo $\varpi$ for any $s$. We prove it by induction on $s$.

Clearly $I_0=(u_0^{\ra})=(X_{0,0})=J_0$.
	Assume $s>0$, write it as $s=ne+i$ with $0\leq i<e$ and assume $(\varpi,I_m)=(\varpi ,J_m)$ for all $m<s$.
	Note that $I_s=(I_{s-1}, u_s^{\ra})$ and $J_s=(J_{s-1}, X_{n,i}^{q^{n(e-1)+i}})$. 
	Hence it is sufficient to prove that 
	\begin{equation}\label{e.claim}
	u_{ne+i}^\ra=u_s^\ra \equiv  \alpha^nX_{n,i}^{q^{n(e-1)+i}}  \quad \text{ mod } (\varpi^{s},J_{s-1}) 
	\end{equation}
since 	$\alpha:=\pi/\varpi^e$ is a unit in $\cO'$.

Note  that \begin{equation}\label{e.forc} \sum_{j=0}^{e-1}\varpi^j \Phi_{s }(X_{\cdot,j})=\Phi_{s}'(u^\ra_0,  \dots)=\Phi_{s-1}'((u_0^\ra)^q,\dots )+\varpi^{s} u_{s}^\ra,\end{equation}
where the first equality follows by the commutativity of diagram \eqref{d.double} and the second one by definition of the polynomials $\Phi'_m$.
	The left hand side of \eqref{e.forc} is sum of monomials of the form
\[\varpi^j \pi^m X_{m,j}^{q^{s-m}}=\varpi^{me+j}\alpha^mX_{m,j}^{q^{s-m}} 
\]
with  $m\leq s=ne+i$ and $ 0\leq j<e$.

	If $m>n$,  the $\varpi$-order of the coefficient is bigger than $s$; similarly if $m=n$ and $j>i$. 
	Hence
	\begin{equation*}\sum_{j=0}^{e-1}\varpi^j \Phi_{s,j}\equiv
	\varpi^{s}\alpha^nX_{n,i}^{q^{s-n}} +\sum_{me+j<s} \varpi^{me+j}\alpha^mX_{m,j}^{q^{s-m}} \quad 
	\text{ mod } (\varpi^{s+1}), \end{equation*}
 and one concludes that
 	\begin{equation}\label{e.equiv1}\sum_{j=0}^{e-1}\varpi^j \Phi_{s,j}\equiv
 \varpi^{s}\alpha^nX_{n,i}^{q^{s-n}}   \quad 
 \text{ mod } (\varpi^{s+1}, J_{s-1}). \end{equation} 
 since $X_{m,j}^{q^{s-m}}=X_{m,j}^{q^{ne+i-m}}$ is a power of $X_{m,j}^{m(e-1)+j}\in J_{s-1}$ when $me+j<ne+i=s$. 
 
 We now discuss the right hand side in \eqref{e.forc}.
\begin{equation}\label{e.equiv2}
	\Phi_{s-1}'((u_0^\ra)^q,\dots )+\varpi^{s} u_{s}^\ra =
	\sum_{l=0}^{s-1}\varpi^l (u_l^\ra)^{q^{s-l}} +\varpi^{s} u_{s}^\ra
\equiv \varpi^{s} u_{s}^\ra \quad\text{ mod } (\varpi^{s+1}, J_{s-1}),
\end{equation} 	
 where the last equivalence follows from the fact that $u^\ra_l\equiv 0 $ modulo $(\varpi, J_{s-1})$ by inductive hypothesis.  
 We conclude then by \eqref{e.forc}, \eqref{e.equiv1} and \eqref{e.equiv2} that
 \[\varpi^{s} u_{s}^\ra \equiv  \varpi^{s}\alpha^nX_{n,i}^{q^{s-n}}  \quad 
 \text{ mod } (\varpi^{s+1}, J_{s-1}),  \] 
 whence claim \eqref{e.claim} is true and the proof is finished.
\end{proof}

\subsubsection{The general case}
Let $\cO^{\un}$ be the maximal unramified extension of $\cO$ in $\cO'$. Then by \eqref{e.ucomp} $u_B=u_{(\cO,\cO'),B}$ is the composition
\begin{equation}\label{e.www} W_\cO(B)\stackrel{u_{(\cO,\cO^{\un}),B}}{\longrightarrow} W_{\cO^{\un}}(B)\stackrel{u_{(\cO^{\un},\cO'),B}}{\longrightarrow} W_{\cO'}(B) 
\end{equation}
and results on $u_B$ are usually deduced by a d\'evissage argument. We see here below an example. 
\begin{lemma}\label{l.wow}
	Set $\cO_0=W(\kk)$ and let $B$ be a reduced (respectively semiperfect, perfect) $\kk$-algebra. Then  the homomorphism \[ r_B\colon  W(B) \otimes_{\cO_0}\cO\to W_{\cO}(B) \] 
	induced by the Drinfeld functor is an injective (respectively, surjective, bijective).  In particular the natural map $  \cO\to W_\cO(\kk)$ is an isomorphism.
\end{lemma}
\begin{proof}
Recall that  $\cO_0/\Z_p$ is unramified, $W(B)=W_{\Z_p}(B)$ and the extension $\cO/\cO_0$ is totally ramified. The homomorphism in the lemma is then the composition 
\begin{equation}\label{e.unra} W(B)\otimes_{\cO_0}\cO
\stackrel{u^{\un}\otimes{\id_\cO}}{\longrightarrow}
W_{\cO_0}(B)\otimes_{\cO_0}\cO 
\stackrel{u^{\ra}}{\longrightarrow} W_\cO(B),\end{equation}	
where $u^{\un}\colon u_{(\Z_p,\cO_0),B}\colon W(B)\to W_{\cO_0}(B)$ and $u^{\ra}=u_{(\cO_0,\cO),B}\otimes {\id_{\cO }}$.
Since   $\cO$ is a free $\cO_0$-module, it suffices to  check the indicated properties for $u^{\un}$ and $u^{\ra}$. These follow by Lemmas  \ref{l.ub} and \ref{l.ue}.    
\end{proof}

\subsubsection{The case $\pi=\varpi^e$}

The description of $u_B$ is particularly nice under the assumption that $\pi=\varpi^e$.
Note that if $\cO'/\cO$ is tamely ramified the hypothesis is satisfied up to enlarging $\cO'$.

\begin{lemma} \label{l.u2}
Let $B$ be a $\kk'$-algebra and assume $\pi=\varpi^e$. Then $u_B\colon W_\cO(B)\to W_{\cO'}(B)$ factors through the subset $W_{\cO',e\N_0}(B)$ consisting of vectors $\bm b=(b_0,\dots)$ such that $b_j=0$ if $e\nmid j$. If $B$ is semiperfect its image is $W_{\cO',e\N_0}(B)$, thus in this hypothesis, $W_{\cO',e\N_0}(B)$ is a subring of $W_{\cO'}(B)$. If $B$ is perfect then $W_{\cO}(B)$ is isomorphic to $W_{\cO',e\N_0}(B)$.
\end{lemma}
\begin{proof}
	By Lemma \ref{l.ub} the case $e=1$ is clear.
	Since $u_B=u_{(\cO,\cO'),B}$ is the composition of the maps in \eqref{e.www} we may assume  that $\cO'/\cO$ is totally ramified. Let $\W_\cO=\spec \ \cO[X_0, \dots]$, $\W_{\cO'}=\spec \ \cO'[Y_0, \dots]$ and set $u_i=u^*(Y_i)$.
It is
\begin{align*}\Phi_{n}(X_0,X_1,\dots)&= \Phi_{n-1}(X_0^q,\dots)+\pi^nX_n  \qquad \text{ for } n\geq 1,\\
	\Phi_{n}^\prime(Y_0,Y_1,\dots)&= \Phi_{n-1}'(Y_0^q,\dots)+\varpi^nY_n   \qquad   \text{ for } n\geq 1, \\
	\Phi'_m(u_0,u_1,\dots)&= \Phi_{m}(X_0,X_1,\dots) \qquad \text{ for } m\geq 0.
	\end{align*}  
 One checks recursively that $u_0=X_0$, $u_i\equiv 0$ mod $(\varpi)$ if $e\nmid i$ and $u_{ne}\equiv X_n^{q^{n(e-1)}}$ mod $(\varpi)$. 
Hence, if $B$ is any $\cO'$-algebra and $\bm b=(b_0,\dots)\in W_\cO(B)$, then $u_B(\bm b)=\bm c=(c_0,c_1,\dots)$ with $c_0=b_0$, $c_{ne}\equiv b_n^{q^{e(n-1)}}$ mod $\varpi B $ and $c_j\in \varpi B$ otherwise. 
In particular, if $B$ is a $\kk'$-algebra, it is $c_{ne}=b_n^{q^{e(n-1)}}$ for any $n\geq 0$ and zero otherwise. This implies that $u_B(W_\cO(B))\subseteq W_{\cO',e\N_0}(B)$ with equality if $B$ is semiperfect. \end{proof}

If $\kk=\kk'$ we have a better understanding of $\bm u^\ra$.

\begin{lemma} 
Let $\cO'/\cO$ be a
totally ramified extension of degree $e$. Assume $\pi=\varpi^e$ and let $B$ be a $\kk'$-algebra. Then the homomorphism \[
u_B^{\ra}\colon W_\cO(B)\otimes_\cO\cO'=\oplus_{i=0}^{e-1} W_\cO(B) \varpi^i \to W_{\cO'}(B), \qquad \sum_{i=0}^{e-1} \bm b_i\varpi^i \mapsto \sum_{i=0}^{e-1} u_B(\bm b_i)\varpi^i, 
\] $\bm b_i\in W_\cO(B)$, maps the module $W_\cO(B)\varpi^i$ into $ W_{\cO',i+e\N_{0}}(B)$ and it is injective (respectively, surjective, bijective) if $B$ is reduced (respectively, semiperfect, perfect). 
\end{lemma}
\begin{proof}
We have seen in Lemma \ref{l.u2} that for any $\bm b\in W_\cO(B)$ it is $u_B(\bm b)\in W_{\cO',e\N_0}(B)$;  hence by \eqref{e.FVVF} and \eqref{e.Fp2} $u_B(\bm b)\varpi=VFu_B(\bm b)\in VW_{\cO',e\N_{0}}(B)=W_{\cO',1+e\N_{0}}(B)$  and recursively $u_B(\bm b)\varpi^i\in W_{\cO',i+e\N_{0}}(B)$.  Note further that the subsets $e\N_{0}, 1+ e\N_{0},\dots, e-1 +e\N_{0}$ form a partition of $\N_0$ so that the sum $\sum_{i=0}^{e-1} u_B(\bm b_i)\varpi^i$ is simply obtained by ``glueing''  the components of each vector $u_B(\bm b_i)\varpi^i=V^iF^iu_B(\bm b_i)$ (see Remark \ref{r.glue}). 
As a consequence the injectivity (respectively, surjectivity) statement follows from Lemma \ref{l.inj}.
\end{proof}

\section{The comparison result}
Let $\cO$ be a complete discrete valuation ring with residue field $\kk$ of cardinality $q=p^h$, and absolute ramification $e$. Set  $\cO_0=W(\kk)$. As seen in Lemma \ref{l.wow}
 we may consider the Drinfeld map $u\colon W(A)\to W_\cO(A)$ for any $\kk$-algebra $A$ and hence we extend it to a natural homomorphism of $\cO$-algebras
\begin{equation*}
r_A:=u\otimes{\id_{\cO}}\colon W(A)\otimes_{W(\kk)}\cO \to W_\cO(A).
\end{equation*}
In other word, due to the description of $A$-sections of $\Rg_\cO$ in \eqref{e.ra}, there exists a morphism of $\kk$-ring schemes
\begin{equation}\label{e.r}\bm r\colon \Rg_\cO\to \W_{\cO,\kk}:=\W_{\cO}\times_{\cO}\spec(\kk) 	\end{equation}
that coincides with $r_A$ on $A$-sections. Then Lemma  \ref{l.wow} can be rewritten as follows.
\begin{lemma}  \label{l.r}
	If $A$ is a reduced (respectively, semiperfect, perfect) $\kk$-algebra then $r_A\colon \Rg_\cO(A)\to \W_\cO(A)$ is injective (respectively, surjective, bijective).  
\end{lemma}  

We can now prove the comparison result announced in the introduction.
\begin{theorem} \label{t.m} 
	The morphism $\bm r\colon \Rg_\cO\to \W_{\cO,\kk}$ defined in \eqref{e.r} induces an isomorphism $\bm r^\pf\colon \Rg_\cO^\pf\to \W_{\cO,\kk}^\pf$ on perfections.
	Hence $\bm r$ is a universal homeomorphism, thus surjective, and it has pro-infinitesimal kernel.  
\end{theorem}
\begin{proof}
By   Lemma \ref{l.r} and \ref{l.pf} the morphism $\bm r^\pf$ is invertible and $\bm r$ is a universal homeomorphism.
	Further, $\bm r$ is a morphism of affine $\kk$-group schemes and
	\[\ker(\bm r)(\bar\kk)= \ker\left(\Rg_\cO(\bar \kk)\to \W_\cO(\bar \kk)\right)\simeq \ker\left(\Rg_\cO^\pf(\bar \kk)\to \W_\cO^\pf(\bar\kk)\right)= \ker(\bm r^\pf)(\bar\kk )=\{0\},\] where $\bar k$ denotes an algebraic closure of $\kk$ and the bijection in the middle follows by \eqref{e.adj}. Hence $\ker(\bm r)$ is proinfinitesimal by \cite[V \S 3 Lemme 1.4]{dg}.
\end{proof}  

We can say something more on the kernel of $ \bm r$.
 
 \begin{lemma} 	\label{l.examples}
 	a) If $\cO=W(\kk)$ then $\bm r=\bm u_{(\Z_p, W(\kk)),\kk}$  and \[\ker \bm r\simeq \spec(\kk[X_0,X_1,\dots]/(X_0, \dots,X_i^{p^{i(h-1)}}\!\!, \dots).\]
 	
 	b) If $\kk=\F_p$,  then $\bm r=\bm u^\ra_{\kk}$ and
 \[\ker \bm r\simeq \spec(\F_p[X_{n,i}; n\in \N_0, 0\leq i<e]/( X_{n,i}^{p^{n(e-1)+i}}  ; n\in \N_0, 0\leq i<e)).\]  
 	
 	c) In general, $\ker(\bm r)$ is extension of a proinfinitesimal group scheme as in Proposition \ref{p.integral} by the product of $e$ proinfinitesimal group schemes as in a).
 \end{lemma} 
\begin{proof} Consider the extension $\cO/\Z_p$. Statements a) and b) follow from   Propositions \ref{p.perf} and \ref{p.integral}. For the general case, note that \eqref{e.unra} implies that $\bm r$, as morphism of $\kk$-group schemes,  is the composition
  	\[   \prod_{i=0}^{e-1} \W_{ \kk}\stackrel{\prod_i\bm u_\kk}{\longrightarrow} \prod_{i=0}^{e-1} \W_{\cO_0,\kk}\stackrel{\bm u^\ra_\kk}{\longrightarrow}  \W_{\cO,\kk}  \]
  	where $\bm u_k$ on the first arrow stays for $\bm u_{(\Z_p,\cO_0), \kk}$, whose kernel was described in a), and $\bm u^\ra_\kk$ is the morphism in Proposition \ref{p.integral} for the ramified extension $\cO/\cO_0$. Hence the conclusion follows.
 \end{proof}

%

 \end{document}